\documentclass[12pt,reqno]{amsart} 
\usepackage{amssymb,amscd,url}

\begin{document}

\allowdisplaybreaks

\newif\ifdraft 
\drafttrue
\draftfalse
\newcommand{\DRAFTNUMBER}{3}
\newcommand{\DATE}{\today\ifdraft(\ Draft \DRAFTNUMBER)\fi}
\newcommand{\TITLE}{Lehmer's Conjecture for Polynomials Satisfying a
  Congruence Divisibility Condition and an Analogue for Elliptic
  Curves}
\newcommand{\TITLERUNNING}{On Lehmer's Conjecture for Polynomials and
  for Elliptic Curves}

\hyphenation{ca-non-i-cal arch-i-me-dean non-arch-i-me-dean}


\newtheorem{theorem}{Theorem}
\newtheorem{lemma}[theorem]{Lemma}
\newtheorem{conjecture}[theorem]{Conjecture}
\newtheorem{proposition}[theorem]{Proposition}
\newtheorem{corollary}[theorem]{Corollary}
\newtheorem*{claim}{Claim}

\theoremstyle{definition}
\newtheorem{question}{Question}
\renewcommand{\thequestion}{\Alph{question}} 
\newtheorem*{definition}{Definition}
\newtheorem{example}[theorem]{Example}
\newtheorem{remark}[theorem]{Remark}

\theoremstyle{remark}
\newtheorem*{acknowledgement}{Acknowledgements}



\newenvironment{notation}[0]{%
  \begin{list}%
    {}%
    {\setlength{\itemindent}{0pt}
     \setlength{\labelwidth}{4\parindent}
     \setlength{\labelsep}{\parindent}
     \setlength{\leftmargin}{5\parindent}
     \setlength{\itemsep}{0pt}
     }%
   }%
  {\end{list}}

\newenvironment{parts}[0]{%
  \begin{list}{}%
    {\setlength{\itemindent}{0pt}
     \setlength{\labelwidth}{1.5\parindent}
     \setlength{\labelsep}{.5\parindent}
     \setlength{\leftmargin}{2\parindent}
     \setlength{\itemsep}{0pt}
     }%
   }%
  {\end{list}}
\newcommand{\Part}[1]{\item[\upshape#1]}
\newcommand{\ProofPart}[1]{\par\noindent\textup{#1}\enspace\ignorespaces}

\renewcommand{\a}{\alpha}
\renewcommand{\b}{\beta}
\newcommand{\g}{\gamma}
\renewcommand{\d}{\delta}
\newcommand{\e}{\epsilon}
\newcommand{\f}{\varphi}
\newcommand{\bfphi}{{\boldsymbol{\f}}}
\renewcommand{\l}{\lambda}
\renewcommand{\k}{\kappa}
\newcommand{\lhat}{\hat\lambda}
\newcommand{\m}{\mu}
\newcommand{\bfmu}{{\boldsymbol{\mu}}}
\renewcommand{\o}{\omega}
\renewcommand{\r}{\rho}
\newcommand{\rbar}{{\bar\rho}}
\newcommand{\s}{\sigma}
\newcommand{\sbar}{{\bar\sigma}}
\renewcommand{\t}{\tau}
\newcommand{\z}{\zeta}

\newcommand{\D}{\Delta}
\newcommand{\G}{\Gamma}
\newcommand{\F}{\Phi}

\newcommand{\ga}{{\mathfrak{a}}}
\newcommand{\gA}{{\mathfrak{A}}}
\newcommand{\gb}{{\mathfrak{b}}}
\newcommand{\gB}{{\mathfrak{B}}}
\newcommand{\gc}{{\mathfrak{c}}}
\newcommand{\gC}{{\mathfrak{C}}}
\newcommand{\gd}{{\mathfrak{d}}}
\newcommand{\gD}{{\mathfrak{D}}}
\newcommand{\gI}{{\mathfrak{I}}}
\newcommand{\gm}{{\mathfrak{m}}}
\newcommand{\gn}{{\mathfrak{n}}}
\newcommand{\go}{{\mathfrak{o}}}
\newcommand{\gO}{{\mathfrak{O}}}
\newcommand{\gp}{{\mathfrak{p}}}
\newcommand{\gP}{{\mathfrak{P}}}
\newcommand{\gq}{{\mathfrak{q}}}
\newcommand{\gR}{{\mathfrak{R}}}

\newcommand{\Abar}{{\bar A}}
\newcommand{\Ebar}{{\bar E}}
\newcommand{\Kbar}{{\bar K}}
\newcommand{\Pbar}{{\bar P}}
\newcommand{\Sbar}{{\bar S}}
\newcommand{\Tbar}{{\bar T}}
\newcommand{\ybar}{{\bar y}}
\newcommand{\phibar}{{\bar\f}}

\newcommand{\Acal}{{\mathcal A}}
\newcommand{\Bcal}{{\mathcal B}}
\newcommand{\Ccal}{{\mathcal C}}
\newcommand{\Dcal}{{\mathcal D}}
\newcommand{\Ecal}{{\mathcal E}}
\newcommand{\Fcal}{{\mathcal F}}
\newcommand{\Gcal}{{\mathcal G}}
\newcommand{\Hcal}{{\mathcal H}}
\newcommand{\Ical}{{\mathcal I}}
\newcommand{\Jcal}{{\mathcal J}}
\newcommand{\Kcal}{{\mathcal K}}
\newcommand{\Lcal}{{\mathcal L}}
\newcommand{\Mcal}{{\mathcal M}}
\newcommand{\Ncal}{{\mathcal N}}
\newcommand{\Ocal}{{\mathcal O}}
\newcommand{\Pcal}{{\mathcal P}}
\newcommand{\Qcal}{{\mathcal Q}}
\newcommand{\Rcal}{{\mathcal R}}
\newcommand{\Scal}{{\mathcal S}}
\newcommand{\Tcal}{{\mathcal T}}
\newcommand{\Ucal}{{\mathcal U}}
\newcommand{\Vcal}{{\mathcal V}}
\newcommand{\Wcal}{{\mathcal W}}
\newcommand{\Xcal}{{\mathcal X}}
\newcommand{\Ycal}{{\mathcal Y}}
\newcommand{\Zcal}{{\mathcal Z}}

\renewcommand{\AA}{\mathbb{A}}
\newcommand{\BB}{\mathbb{B}}
\newcommand{\CC}{\mathbb{C}}
\newcommand{\FF}{\mathbb{F}}
\newcommand{\GG}{\mathbb{G}}
\newcommand{\NN}{\mathbb{N}}
\newcommand{\PP}{\mathbb{P}}
\newcommand{\QQ}{\mathbb{Q}}
\newcommand{\RR}{\mathbb{R}}
\newcommand{\ZZ}{\mathbb{Z}}

\newcommand{\bfa}{{\mathbf a}}
\newcommand{\bfb}{{\mathbf b}}
\newcommand{\bfc}{{\mathbf c}}
\newcommand{\bfe}{{\mathbf e}}
\newcommand{\bff}{{\mathbf f}}
\newcommand{\bfg}{{\mathbf g}}
\newcommand{\bfp}{{\mathbf p}}
\newcommand{\bfr}{{\mathbf r}}
\newcommand{\bfs}{{\mathbf s}}
\newcommand{\bft}{{\mathbf t}}
\newcommand{\bfu}{{\mathbf u}}
\newcommand{\bfv}{{\mathbf v}}
\newcommand{\bfw}{{\mathbf w}}
\newcommand{\bfx}{{\mathbf x}}
\newcommand{\bfy}{{\mathbf y}}
\newcommand{\bfz}{{\mathbf z}}
\newcommand{\bfA}{{\mathbf A}}
\newcommand{\bfF}{{\mathbf F}}
\newcommand{\bfB}{{\mathbf B}}
\newcommand{\bfD}{{\mathbf D}}
\newcommand{\bfG}{{\mathbf G}}
\newcommand{\bfI}{{\mathbf I}}
\newcommand{\bfM}{{\mathbf M}}
\newcommand{\bfzero}{{\boldsymbol{0}}}

\newcommand{\Adele}{\textsf{\upshape A}}
\newcommand{\Aut}{\operatorname{Aut}}
\newcommand{\bad}{{\textup{bad}}}  
\newcommand{\Bern}{\mathbb{B}} 
\newcommand{\Br}{\operatorname{Br}}  
\newcommand{\Cond}{\mathfrak{F}}
\newcommand{\Disc}{\mathfrak{D}}
\newcommand{\ds}{\displaystyle}
\newcommand{\dsum}{\displaystyle\sum}
\newcommand{\Div}{\operatorname{Div}}
\renewcommand{\div}{{\operatorname{div}}}
\newcommand{\End}{\operatorname{End}}
\newcommand{\Fbar}{{\bar{F}}}
\newcommand{\Gal}{\operatorname{Gal}}
\newcommand{\GL}{\operatorname{GL}}
\newcommand{\Index}{\operatorname{Index}}
\newcommand{\Image}{\operatorname{Image}}
\newcommand{\hhat}{{\hat h}}
\newcommand{\Ker}{{\operatorname{ker}}}
\newcommand{\Mabs}{\mathcal{M}}
\newcommand{\MOD}[1]{~(\textup{mod}~#1)}
\newcommand{\Norm}{{\operatorname{\mathsf{N}}}}
\newcommand{\notdivide}{\nmid}
\newcommand{\normalsubgroup}{\triangleleft}
\newcommand{\odd}{{\operatorname{odd}}}
\newcommand{\onto}{\twoheadrightarrow}
\newcommand{\Orbit}{\mathcal{O}}
\newcommand{\ord}{\operatorname{ord}}
\newcommand{\Per}{\operatorname{Per}}
\newcommand{\PrePer}{\operatorname{PrePer}}
\newcommand{\PGL}{\operatorname{PGL}}
\newcommand{\Pic}{\operatorname{Pic}}
\newcommand{\Prob}{\operatorname{Prob}}
\newcommand{\Qbar}{{\bar{\QQ}}}
\newcommand{\rank}{\operatorname{rank}}
\renewcommand{\Re}{\operatorname{Re}}
\newcommand{\Resultant}{\operatorname{Res}}
\renewcommand{\setminus}{\smallsetminus}
\newcommand{\SL}{\operatorname{SL}}
\newcommand{\Span}{\operatorname{Span}}
\newcommand{\Spec}{\operatorname{Spec}}
\newcommand{\tors}{{\textup{tors}}}
\newcommand{\Trace}{\operatorname{Trace}}
\newcommand{\twistedtimes}{\mathbin{%
   \mbox{$\vrule height 6pt depth0pt width.5pt\hspace{-2.2pt}\times$}}}
\newcommand{\UHP}{{\mathfrak{h}}}    
\newcommand{\<}{\langle}
\renewcommand{\>}{\rangle}

\newcommand{\longhookrightarrow}{\lhook\joinrel\longrightarrow}
\newcommand{\longonto}{\relbar\joinrel\twoheadrightarrow}

\newcounter{CaseCount}
\Alph{CaseCount}
\def\Case#1{\par\vspace{1\jot}\noindent
\stepcounter{CaseCount}
\framebox{Case \Alph{CaseCount}.\enspace#1}
\par\vspace{1\jot}\noindent\ignorespaces}

\title[\TITLERUNNING]{\TITLE}
\date{\DATE}

\author[Joseph H. Silverman]{Joseph H. Silverman}
\email{jhs@math.brown.edu}
\address{Mathematics Department, Box 1917
         Brown University, Providence, RI 02912 USA}
\thanks{The author's research partially supported by NSF grants DMS-0650017 and
DMS-0854755.}

\begin{abstract}
A number of authors have proven explicit versions of Lehmer's
conjecture for polynomials whose coefficients are all congruent to~$1$
modulo~$m$. We prove a similar result for polynomials~$f(X)$ that are
divisible in~$(\ZZ/m\ZZ)[X]$ by a polynomial of the form
$1+X+\cdots+X^n$ for some $n\ge\e \deg(f)$. We also formulate and
prove an analogous statement for elliptic curves.
\end{abstract}



\maketitle

\section*{Introduction}

Let
\[
  h:\Qbar\longrightarrow[0,\infty)
\]
denote the absolute logarithmic height~\cite{HiSi,LangDG}. Lehmer's
conjecture~\cite{Leh} asserts that there is an absolute constant $C>0$
such that if~$f(X)\in\ZZ[X]$ is a monic polynomial of degree~$D\ge1$
whose roots are not roots of unity, then
\begin{equation}
  \label{eqn:lehmerconj}
  \sum_{f(\a)=0} h(\a) \ge C.
\end{equation}
This problem has a long history; see for
example~\cite{BlMo,Leh,Smy2,Smy1,Stew}.  The best general result
known, which is due to Dobrowolski~\cite{Dob}, says that $\sum
h(\a)\ge C(\log\log D/\log D)^3$.  Various authors have considered
Lehmer's problem for restricted values of~$\a$. For example, Amoroso
and Dvornicich~\cite{AmDv} show that if the roots of~$f(X)$ generate
an abelian extension of~$\QQ$, then $\sum h(\a)\ge D(\log5)/12$.  
\par
An interesting class of polynomials are those whose coefficients are
all odd. More generally, one can consider polynomials whose
coefficients are congruent to~$1$ modulo~$m$, as in the following
result.

\begin{theorem}
\label{thm:lehmodm}
\textup{(Borwein, Dobrowolski, Mossinghoff \cite{BoDoMo})} 
Let $m\ge2$, and let $f(X)\in\ZZ[X]$ be a monic polynomial of
degree~$D$ that satisfies
\begin{equation}
\label{eqn:feqaop}
   f(X) \equiv X^D+X^{D-1}+\cdots+X^2+X+1 \pmod{m}.
\end{equation}
Then
\[
  \sum_{f(\a)=0} h(\a) \ge \frac{D}{D+1}C_m,
\]
where we may take
\[
  C_2 = \frac{1}{4}\log5
  \quad\text{and}\quad
  C_m = \log\frac{\sqrt{m^2+1}}{2}\quad\text{for $m\ge3$.}
\]
\end{theorem}

We mention that an earlier paper~\cite{BoHaMo} does the case of
non-reciprocal polynomials, and subsequent papers
\cite{DuMo,GaIsMoPiWi} give improved values for~$C_m$, although
asymptotically they all have the form $C_m=\log(m/2) + O(1/m^2)$.
We also note the papers~\cite{MR2313990,MR2496463} which give various
generalizations of Theorem~\ref{thm:lehmodm}, including weakening the
congruence condition~\eqref{eqn:feqaop}, working over number fields,
and considering heights of points and subspaces in projective space.

Our first result is the following generalization of
Theorem~\ref{thm:lehmodm}, albeit with less sharp constants.  See
Theorem~\ref{thm:hageCDAmnlogm} and Corollary~\ref{cor:htgeelogm} for
our precise results.

\begin{theorem}
\label{thm:intromultgp}
For all $\e>0$ there is a constant $C_\e>0$ with the following
property\textup: Let $f(X)\in\ZZ[X]$ be a monic polynomial of
degree~$D$ such that
\begin{equation}
  \label{eqn:fXdivXn1X1ZmZ}
  \text{$f(X)$ is divisible by $X^{n-1}+X^{n-2}+\cdots+X+1$ in $(\ZZ/m\ZZ)[X]$}
\end{equation}
for some integers
\[
 m\ge2\qquad\text{and}\qquad n\ge\max\{\e D,2\}.
\]
Suppose further that no root of~$f(X)$ is a root of unity. Then
\[
  \sum_{f(\a)=0} h(\a) \ge C_\e\log m.
\]
In particular, Lehmer's conjecture~\eqref{eqn:lehmerconj} is true for this
class of polynomials.
\end{theorem}

The elliptic analogue of Lehmer's conjecture says that if~$E/K$
is an elliptic curve defined over a number field, then there
is a constant $C_{E/K}>0$ such that for all nontorsion 
points $Q\in E(\Kbar)$ of degree $D_Q=[K(Q):K]$ we have
\begin{equation}
  \label{eqn:ellipticlehmer}
  D_Q \hhat_E(Q) \ge C_{E/K}.
\end{equation}
Here~$\hhat_E$ is the logarithmic canonical height on~$E$.  There has
been considerable work on the elliptic Lehmer conjecture; see for
example~\cite{HiSi2,Laur,Mass}.  Our second main result is an elliptic
analogue of Theorem~\ref{thm:intromultgp}.  An initial difficulty is
to find an appropriate elliptic version of the mod~$m$ divisibility
condition~\eqref{eqn:fXdivXn1X1ZmZ}. In Section~\ref{section:prop1} we
show that~\eqref{eqn:fXdivXn1X1ZmZ} implies a lower bound for a
certain sum over the roots of~$f$, and it is this weaker property that
we generalize and adapt to the elliptic setting. Using this
definition, we are able to prove the following result.  (See
Corollary~\ref{cor:ellipticlehmer} for a precise statement.)

\begin{theorem}
\label{thm:mainthmintro}
Let $K/\QQ$ be a number field and let~$E/K$ be an elliptic curve.  Fix
some $\e>0$.  Then the elliptic Lehmer
conjecture~\eqref{eqn:ellipticlehmer} is true for all points~$Q\in
E(\Kbar)$ satisfying an elliptic mod~$\gm$ condition analogous
to~\eqref{eqn:fXdivXn1X1ZmZ} for some ideal~$\gm$ such that
$\Norm_{K/\QQ}\gm\ge2$ and such that~$E$ has good reduction at all
primes dividing~$\gm$ and for some $n\ge\max\{\e{D_Q},2\}$.  The
lower bound in~\eqref{eqn:ellipticlehmer} will have the form
$C_{E/K,\e}\log m$.
\end{theorem}

Theorem~\ref{thm:intromultgp} deals with congruences related to
cyclotomic polynomials, which is natural when studying Lehmer's
problem, but one might consider other sorts of congruence
conditions. For example, suppose that~$f(X)$ is congruent modulo~$m$
to \text{$X^D+X^{D-1}+\cdots+X^2+X-1$}.  Samuels~\cite{MR2313990} has
considered general conditions of this sort. In
Section~\ref{section:coefcong} we briefly reprove one of Samuels'
results and use it to make a number of remarks concerning possible
generalizations. 

\begin{remark}
Our Theorem~\ref{thm:intromultgp} and some of Samuels' principal
results~\cite{MR2313990} are height bounds for polynomials satisfying
various sorts of congruence conditions, so we conclude this
introduction by briefly describing how the results differ. Our
polynomials satisfy a divisibility condition modulo~$m$, so
multiplying by~$X-1$, our theorem applies to polynomials~$F(X)$ of the
form
\[
  F(X) = (X^n-1)A(X) + mB(X)\qquad\text{for some $A,B\in\ZZ[X]$.}
\]
In general, we obtain a bound for all~$n$ (see
Lemma~\ref{lemma:htgeelogm}), and in particular we prove Lehmer's
conjecture if~$n\ge\e\deg(F)$. The results in~\cite{MR2313990} apply
to (factors of) polynomials that are congruent modulo~$m$ to a simpler
polynomial of the same degree. For example, a typical result
in~\cite{MR2313990} is a bound for (noncyclotomic factors of)
polynomials~$F(X)$ of degree~$nr$ satisfying
\[
  F(X) = (X^n-1)^r + mB(X)\qquad\text{for some $B\in\ZZ[X]$.}
\]
Thus although there is some overlap, our result and the results
in~\cite{MR2313990} apply to largely different classes of polynomials.
It might be interesting to combine the methods of the two papers to
prove a general result encompassing both.
\end{remark}

\begin{acknowledgement}
The author would like to thank Michael Mossinghoff for introducing him
to the topic of polynomials whose coefficients satisfy congruence
conditions and both Michael and Igor Shparlinski for their helpful
comments on an initial draft of the paper.
\end{acknowledgement}

\section{A Reformulation of Property~$\eqref{eqn:feqaop}$}
\label{section:prop1}

We start by normalizing our absolute values.

\begin{definition}
We let $M_\QQ$ be the usual collection of absolute values on~$\QQ$,
and for any algebraic extension~$K/\QQ$, we write~$M_K$ for the set of
all extensions of these absolute values to~$K$.  For~$\a\in\Qbar$
and~$v\in M_\Qbar$, we define a normalized absolute value by choosing
a finite extension~$K/\QQ$ with~$\a\in K$ and setting
\[
  \|\a\|_v = |\a|_v^{[K_v:\QQ_v]/[K:\QQ]}.
\]
We also define a normalized valuation by
\[
  v(\a) = -\log\|\a\|_v.
\]
Then the absolute logarithmic height of~$\a$ is defined by
\[
  h(\a) = \sum_{v\in M_K} \max\bigl\{\log\|\a\|_v,0\bigr\}.
\]
We write $M_K^0$, respectively~$M_K^\infty$, for the set of
nonarchimedean, respectively archimedean, absolute values in~$M_K$.
\end{definition}

\begin{remark}
With the above normalization we have the product formula
\[
  \prod_{v\in M_K} \|\a\|_v = 1\qquad\text{for all $\a\in K^*$.}
\]
In particular, if~$\a\in K$ is a nonzero algebraic integer, then
\[
  h(\a) = \sum_{v\in M_K^\infty}\max\bigl\{\log\|\a\|_v,0\bigr\}
    = \sum_{v\in M_K^0} v(\a).
\]
We also remark that
\[
  \prod_{v\in M_K^\infty} \|\a\|_v
  = \prod_{v\in M_K^0} \|\a\|_v^{-1} = \Norm_{K/\QQ}(\a)^{1/[K:\QQ]}.
\]
\end{remark}

\begin{remark}
Let $f(X)\in\ZZ[X]$ be a monic polynomial. Then the classical Mahler
measure~$M(f)$ of~$f$ is related to the heights of its roots via the
formula
\[
  \log M(f) = \sum_{f(\a)=0} h(\a).
\]
In this paper we use the ``sum the heights'' notation because it has
an obvious generalization to other algebraic groups such as elliptic
curves. In such sums, we always include the roots of~$f$ with their
multiplicities.
\end{remark}

\begin{definition}
For notational convenience, we let
\[
  \F_n(X) = X^n+X^{n-1}+\cdots+X+1.
\]
If~$n$ is prime, this is the usual cyclotomic polynomial; in general it
is a product of classical cyclotomic polynomials.
\end{definition}

Property~\eqref{eqn:feqaop} of Theorem~\ref{thm:lehmodm} says that all
of the coefficients of the polynomial~$f$ are congruent to~$1$
modulo~$m$.  This is equivalent to saying that the monic degree~$D$
polynomial~$f(X)\in\ZZ[X]$ is divisible by~$\F_D(X)$ in the
ring~$(\ZZ/m\ZZ)[X]$. Note that although this ring will contain zero
divisors if~$m$ is composite, divisibility by a monic polynomial is
still a well-behaved property.  The next proposition gives some
properties that are weaker than Property~\eqref{eqn:feqaop} of
Theorem~\ref{thm:lehmodm}.

\begin{proposition}
\label{prop:wkmodm}
Let $m,n\ge2$, and let~$f(X)\in\ZZ[X]$ be a monic polynomial of
degree~$D$.  Consider the following three conditions\textup:
\begin{parts}
\Part{(i)}
  \text{$f(X)$ is divisible by $\F_{n-1}(X)$ in $(\ZZ/m\ZZ)[X]$.}
\vspace{1\jot}
\Part{(ii)}
  $m^{n-1} \mid \Resultant\bigl(f(X),\F_{n-1}(X)\bigr)$.
\vspace{1\jot}
\Part{(iii)}
  $\displaystyle \sum_{v\mid m} \sum_{f(\a)=0}
    v(\a^{n}-1) \ge (n-1)\log m$.
\end{parts}
Then
\[
  \textup{(i)} \Longrightarrow
  \textup{(ii)} \Longrightarrow
  \textup{(iii)}.
\]
\textup(In~\textup{(iii)}, we may work over any field in which~$f$ factors
completely. The way that we have normalized our absolute values
ensures that the choice of field does not matter.\textup) 
\end{proposition}

\begin{proof}
Property~(i) says that
\[
  f(X) = \F_{n-1}(X)A(X) + mB(X)
  \quad\text{for some $A(X),B(X)\in\ZZ[X]$.}
\]
This implies that
\begin{align*}
  \Resultant\bigl(f(X),\F_{n-1}(X)\bigr)
  &= \Resultant\bigl(\F_{n-1}(X)A(X) + mB(X),\F_{n-1}(X)\bigr)\\
  &= \Resultant\bigl(mB(X),\F_{n-1}(X)\bigr)\\
  &= m^{n-1} \Resultant\bigl(B(X),\F_{n-1}(X)\bigr).
\end{align*}
Thus~(ii) is true. 
\par
We next prove that~(ii) implies~(iii).  For any non-archimedean
absolute value~$v$ we have
\begin{align}
  \label{eqn:resgemD}
  \bigl\|\Resultant\bigl(f(X),x^{n}-1\bigr)\bigr\|_v
  &= \bigl\|\Resultant\bigl(f(X),\F_{n-1}(X)\bigr)\bigr\|_v
     \bigl\|f(1)\bigr\|_v  \notag\\*
  &\le \|m\|_v^{n-1}.
\end{align}
\par
A standard formula for the resultant~\cite[Section~V.10]{lang:algebra}
is
\begin{equation}
  \label{eqn:resproddef}
  \Resultant\bigl(f(X),X^{n}-1\bigr)
  = \prod_{f(\a)=0} (\a^{n}-1).
\end{equation}
We take the $v$-absolute value of~\eqref{eqn:resproddef},
use~\eqref{eqn:resgemD}, and multiply over all~$v\mid m$ to obtain the
estimate
\begin{equation}
  \label{eqnvmiDaizv}
  \prod_{v\mid m} \prod_{f(\a)=0} \|\a^{n}-1\|_v
  \le \prod_{v\mid m} \|m\|_v^{n-1} = m^{-(n-1)}.
\end{equation}
Taking~$-\log(\,\cdot\,)$ gives~(iii).  
\end{proof}

We now define a quantiy that generalizes the sum appearing in
Property~(iii) of Proposition~\ref{prop:wkmodm}.

\begin{definition}
Let~$\Acal\subset\Qbar^*$ be a finite set of algebraic integers and
let~$m$ and~$n$ be positive integers. We define
\[
  \Delta(\Acal,m,n) 
  = \sum_{\a\in\Acal}  \frac{1}{\log m} \sum_{v\mid m} \frac{1}{n} v(\a^n-1).
\]
The inner sum is over all $v\in M_K^0$ with $v\mid m$, where~$K$ is
any number field containing~$\Acal$. Our normalization of the
valuations in~$M_\Qbar$ implies that the sum is independent of the
choice of~$K$.
\end{definition}

\begin{proposition}
\label{prop:DAjmngeDAmn}
Let $\Acal$ be a finite set of algebraic integers,
let $j,m,n\ge1$ be rational integers, and 
let $\Acal^j=\{\a^j:\a\in\Acal\}$. Then
\[
  \D(\Acal^j,m,n) \ge \D(\Acal,m,n).
\]
\end{proposition}
\begin{proof}
From the factorization
\[
  X^{jn}-1 = (X^n-1)\F_{j-1}(X^n),
\]
we see that
\[
  v(\a^{jn}-1)=v(\a^{n}-1)+v(\F_{j-1}(\a^n))\ge v(\a^{n}-1)
\]
for any algebraic integer~$\a$ and any nonarchimedean absolute
value~$v$. Summing over $\a\in\Acal$ and $v\mid m$, and then
dividing by~$n\log m$, gives the desired result.
\end{proof}

\begin{remark}
\label{remark:DAfmD1geDD1}
We observe that if $f(X)\in\ZZ[X]$ is a monic polynomial of degree~$D$
and if we write~$\Acal_f$ for the set of roots of~$f(X)$, then
Property~(iii) of Proposition~\ref{prop:wkmodm} can be succintly
written as
\begin{equation}
  \label{eqn:DAfmD1geDD1}
  \Delta(\Acal_f,m,n) \ge \frac{n-1}{n}.
\end{equation}
In particular, if $f$ satisfies the congruence
\[
  f(X)  \equiv \F_D(X) \pmod{m}
\]
as in the statement of Theorem~\ref{thm:lehmodm}, then
\[
  \Delta(\Acal_f,m,D+1) \ge \frac{D}{D+1}.
\]
\end{remark}

\section{A Height Bound for Polynomials Satisfying Congruence Conditions}
\label{section:htbdforpolys}

The next theorem is our main result for number fields.  As
we will see, it generalizes~\cite{BoDoMo,BoHaMo,DuMo,GaIsMoPiWi}
(Theorem~\ref{thm:lehmodm}), albeit with worse constants.  Later we
will prove an elliptic curve version of this theorem and its
consequences.

\begin{theorem}
\label{thm:hageCDAmnlogm}
Let $\Acal\subset\bar\QQ$ be a finite set of algebraic integers, that
does not contain any roots of unity and let~$m$ and~$n$ be positive
integers. Then for all integers $J\ge1$ we have
\begin{equation}
  \label{eqn:sumhageDAmnlogm}
  \sum_{\a\in\Acal}h(\a) \ge \frac{3}{J+2}\left(
  \D(\Acal,m,n) \log m
    - \frac{|\Acal|}{n}\frac{\log(J/2+1)+1}{J}\right).
\end{equation}
\end{theorem}

\begin{remark}
It is always possible to choose a value of~$J$ to obtain a nontrivial,
i.e., positive, lower bound in~\eqref{eqn:sumhageDAmnlogm}. The
optimal choice of~$J$ depends on the relative sizes of
$\D(\Acal,m,n)\log m$ and $|\Acal|/n$. In the application most closely related
to Theorem~\ref{thm:lehmodm}, we have
\[
  n=D+1\qquad\text{and}\qquad
  \D(\Acal,m,n)\log m = \frac{|\Acal|}{n} = \frac{D}{D+1},
\]
so we get
\[
  \sum_{f(\a)=0} h(\a) \ge \frac{D}{D+1}\cdot\frac{3}{J+2}
    \cdot\left(\log m - \frac{\log(J/2+1)+1}{J}\right).
\]
If $m\ge5$, then we obtain a nontrivial lower bound with $J=1$, while
for~$3\le m\le 4$ we need to take $J=2$, and for~$m=2$ we must
take~$J=3$.  Of course, the bound that we obtain is not sharp. 
However, our goal is not to get sharp bounds in this particular case,
where other authors~\cite{BoDoMo,BoHaMo,DuMo,GaIsMoPiWi} have used
intricate techniques to obtain better bounds than we could obtain even
if we took more care. Instead, we aim to show how to obtain nontrivial
bounds that, among other things, imply that Lehmer's conjecture is
true for an interesting class of polynomials that is larger than the
class considered in~\cite{BoDoMo,BoHaMo,DuMo,GaIsMoPiWi}.
\end{remark}

The proof of Theorem~\ref{thm:hageCDAmnlogm}  uses the following
standard Fej\'er kernel estimate, whose proof we relegate to 
Section~\ref{section:fejarproof}.

\begin{proposition}
\label{prop:fejarestimate}
For all $J\ge1$ we have
\[
  \sup_{\substack{z\in\CC\\ |z|\le 1\\}} 
    \sum_{j=1}^J \left(1-\frac{j}{J+1}\right)\log|1-z^j|
  \le \frac{1}{2}\log\left(\frac{J}{2}+1\right)+\frac{1}{2}.
\]
\end{proposition}

\begin{proof}[Proof of Theorem~$\ref{thm:hageCDAmnlogm}$]
Let~$K$ be a number field such that $\Acal\subset K$.
For~$\a\in\Acal$ and~$v\in M_K$, we let
\[
  \a_v = \begin{cases}
     \a&\text{if $\|\a\|_v\le1$,}\\
     \a^{-1}&\text{if $\|\a\|_v>1$,}\\
  \end{cases}
\]
so in particular $\|\a_v\|_v\le 1$.  We now compute
\begin{align}
  \label{eqn:nlogmDlehlogan1v}
  (n\log m)\D(\Acal,m,n) 
  &= \sum_{\a\in\Acal}\sum_{v\mid m} v(\a^n-1)
    &&\text{def. of $\D(\Acal,m,n)$,} \notag\\
  &\le \sum_{\a\in\Acal}\sum_{v\in M_K^0} v(\a^n-1)
    &&\text{since $\a\in\bar\ZZ$,} \notag\\
  &= \sum_{\a\in\Acal}\sum_{v\in M_K^\infty} \log\|\a^n-1\|_v
    &&\text{product rule,} \notag\\
  &\le \sum_{\a\in\Acal}\sum_{\substack{v\in M_K^\infty\\\|\a\|_v>1\\}} \log\|\a^n\|_v
    +  \sum_{\a\in\Acal}\sum_{v\in M_K^\infty} \log\|\a_v^n-1\|_v \hidewidth \notag\\
  &= n\sum_{\a\in\Acal} h(\a) 
    +  \sum_{\a\in\Acal}\sum_{v\in M_K^\infty} \log\|\a_v^n-1\|_v. \hidewidth 
\end{align}
We now replace~$\Acal$ with~$\Acal^j$. Then using $h(\a^j)=jh(\a)$ and
Proposition~\ref{prop:DAjmngeDAmn}, which says that
$\D(\Acal,m,n)\le\D(\Acal^j,m,n)$, we find that
\[
  (n\log m)\D(\Acal,m,n) 
  \le nj\sum_{\a\in\Acal}h(\a) 
    +  \sum_{\a\in\Acal}\sum_{v\in M_K^\infty} \log\|\a_v^{jn}-1\|_v .
\]
We multiply by the Fej\'er multiplier $1-\frac{j}{J+1}$  and sum over
$1\le j\le J$ to obtain
\begin{multline*}
  \frac{Jn\log m}{2}\D(\Acal,m,n) 
  \le \frac{(J^2+2J)n}{6}\sum_{\a\in\Acal}h(\a)  \\*
    +  \sum_{\a\in\Acal}\sum_{v\in M_K^\infty} 
         \sum_{j=1}^J \left(1-\frac{j}{J+1}\right)\log\|\a_v^{jn}-1\|_v .
\end{multline*}
Note that the sum over~$v$ is over archimedean absolute values, so if
we assume that~$\a_v$ is chosen in the unit disk to maximize the
innermost sum over~$j$, we get the estimate
\begin{multline*}
  \frac{Jn\log m}{2}\D(\Acal,m,n)  \\*
  \le \frac{(J^2+2J)n}{6}\sum_{\a\in\Acal}h(\a) 
    +  |\Acal| \sup_{\substack{z\in\CC\\|z|\le1\\}}
         \sum_{j=1}^J \left(1-\frac{j}{J+1}\right)\log|z^j-1|.
\end{multline*}
We can now use Proposition~\ref{prop:fejarestimate} to conclude
that
\[
  \frac{Jn\log m}{2}\D(\Acal,m,n) 
  \le \frac{(J^2+2J)n}{6}\sum_{\a\in\Acal}h(\a) 
    + \frac{|\Acal|}{2}\left(\log\left(\frac{J}{2}+1\right)+1\right).
\]
After a little bit of algebra, we obtain the  desired result.
\end{proof}

We now use our main theorem to prove that Lehmer's conjecture
is true for a certain interesting collection of polynomials.

\begin{corollary}
\label{cor:htgeelogm}
Fix $0<\e\le 1$. Then Lehmer's conjecture~\eqref{eqn:lehmerconj} is
true for the set of polynomials
\begin{equation} 
  \label{eqn:monicmnedegf}
  \left\{ f(X)\in \ZZ[X] :
    \begin{tabular}{@{}l@{}}
       $f(X)$ is monic, its roots are not roots\\ 
       of unity, and there exist integers $m\ge2$ \\
       and $n\ge\max\{2,\e\deg(f)\}$ such that\\
       $\F_{n-1}(X)$ divides $f(X)$ in $(\ZZ/m\ZZ)[X]$\\
    \end{tabular}
  \right\}.
\end{equation}
More precisely, if~$f(X)$ is in the set~\eqref{eqn:monicmnedegf}, then
\begin{equation}
  \label{eqn:logm185e}
  \sum_{f(\a)=0} h(\a) \ge 
   \frac{\log m}{185\e^{-1}\log(24\e^{-1})}.
\end{equation}
\end{corollary}

\begin{remark}
Igor Shparlinski has pointed out that for large~$m$, we may take
$\e=(\log\log m)/(\log m)$ and conclude that if $\F_{n-1}(X)$ divides
$f(X)$ in $(\ZZ/m\ZZ)[X]$ for some
$n\ge((\log\log{m})/(\log{m}))(\deg{f})$, then
\[
  \sum_{f(\a)=0} h(\a) \ge \frac{1}{185}+
    O\left(\frac{\log\log\log m}{\log\log m}\right),
\]
where the big-$O$ constant is absolute.
\end{remark}

The proof of the corollary uses a combination of
Theorem~\ref{thm:hageCDAmnlogm} and Proposition~\ref{prop:wkmodm} as
reformulated in Remark~\ref{remark:DAfmD1geDD1}.  We state the exact
result that we require as a lemma.

\begin{lemma}
\label{lemma:htgeelogm}
Let $f(X)\in\ZZ[X]$ be a monic polynomial of degree~$D$ whose roots
are not roots of unity, let $m,n\ge2$ be integers, and
suppose that~$\F_{n-1}(X)$ divides $f(X)$ in $(\ZZ/m\ZZ)[X]$. Then
\[
  \sum_{f(\a)=0}h(\a) \ge \begin{cases}
    (\log m)/264 &\text{if $\log m\ge D/16n$,} \\[1\jot]
    \dfrac{\log m}{(128D/n\log m)\log(16D/n\log m)}
          &\text{if $\log m\le D/16n$.} \\
  \end{cases}
\]
\end{lemma}
\begin{proof}
Let~$\Acal_f$ be the set of roots of~$f$.  As noted in
Remark~\ref{remark:DAfmD1geDD1}, the divisibility condition on~$f$
implies that $\D(\Acal_f,m,n)\ge (n-1)/n$.  Substituting this
into~\eqref{eqn:sumhageDAmnlogm} of Theorem~\ref{thm:hageCDAmnlogm}
and using~$|\Acal_f|=D$, we find that for all integers $J\ge1$ we have
\begin{align*}
  \sum_{f(\a)=0}h(\a)
  &\ge \frac{3}{J+2}
      \left(\frac{n-1}{n}\log m - \frac{D}{n}\frac{\log(J/2+1)+1}{J}\right).
\end{align*}
Since we are not concerned with optimizing our constants, we observe
that for $n\ge2$ and $J\ge2$, this implies that
\begin{equation}
  \label{eqn:1J12m4DnlogJJ}
  \sum_{f(\a)=0}h(\a)
  \ge \frac{1}{J}
      \left(\frac{1}{2}\log m - \frac{4D}{n}\frac{\log(J)}{J}\right).
\end{equation}
We now consider two cases. First, if~$m$ is large, say
\[
  \log m \ge D/16n,
\]
then taking~$J=57$ gives
\begin{equation}
  \label{eqn:bd1}
  \sum_{f(\a)=0}h(\a) 
     \ge \frac{1}{J}\left(\frac12-4\frac{\log(J)}{J}\right)\log m
     \ge \frac{\log m}{264}.
\end{equation}
Second, suppose that~$m$ is small,
\[
  \log m \le D/16n.
\]
Then we want to choose~$J$ to be an integer satisfying
\begin{equation}
  \label{eqn:JlogJge16dnlogm}
  \frac{J}{\log J} \ge \frac{16D}{n\log m}.
\end{equation}
In particular, since $16D/n\log m\ge256$, it suffices
to take
\[
  J = \left\lfloor \frac{32D}{n\log m}\log\left(\frac{16D}{n\log m}\right)
     \right\rfloor - 1.
\]
Substituting~\eqref{eqn:JlogJge16dnlogm}
into~\eqref{eqn:1J12m4DnlogJJ} and adjusting the constants yields
\begin{equation}
  \label{eqn:bd2}
  \sum_{f(\a)=0}h(\a) \ge \frac{\log m}{4J} \ge
    \frac{\log m}{(128D/n\log m)\log(16D/n\log m)}.
\end{equation}
Combining~\eqref{eqn:bd1} and~\eqref{eqn:bd2} completes the proof of
Lemma~\ref{lemma:htgeelogm}.
\end{proof}

\begin{proof}[Proof of Corollary~$\ref{cor:htgeelogm}$]
We are given that $n\ge\max\{\e D,2\}$. If $\log m\ge D/16n$,
then Lemma~\ref{lemma:htgeelogm} says that
\[
  \sum_{f(\a)=0}h(\a) \ge \frac{\log m}{264}.
\]
This is stronger than~\eqref{eqn:logm185e}, since we have assumed that
$\e\le1$, so we are reduced to the case that~$\log m\le D/16n$.  Since
$n\ge\e D$, this implies that
\[
  \frac{D}{n\log m}\le\frac{D}{\e D\log m} \le \frac{1}{\e\log m},
\]
where the upper bound is at least~$16$. Substituting this into
Lemma~\ref{lemma:htgeelogm}, we find that
\[
  \sum_{f(\a)=0}h(\a) \ge
     \frac{\log m}{(128/\e\log m)\log(16/\e\log m)}.
\]
Since~$m\ge2$, this gives something slightly stronger than the desired
result.
\end{proof}

\section{An Elliptic Analogue of $\D(\Acal,m,n)$}
\label{section:ellmodm}

An amalgamation of Proposition~\ref{prop:wkmodm},
Remark~\ref{remark:DAfmD1geDD1}, and Theorem~\ref{thm:hageCDAmnlogm}
says that if~$f(X)\in\ZZ[X]$ is a monic polynomial of degree~$D$ whose
roots are not roots of unity, then
\begin{align*}
  \left(\begin{tabular}{@{}l@{}}
     $f(X)$ has coefficients\\
     congruent to 1 modulo $m$\\
     \end{tabular}\right)
  &\quad\Longrightarrow\quad
     \D(\Acal_f,m,D+1) \ge \frac{D}{D+1} \\
  &\quad\Longrightarrow\quad
     \sum_{f(\a)=0} h(\a) \ge \frac{D}{D+1}C_m.
\end{align*}
The key estimate is Theorem~\ref{thm:hageCDAmnlogm}, which gives a
general lower bound for~$\sum h(\a)$ in terms of $\D(\Acal,m,n)$.
In this section we define an elliptic analogue of the quantity
$\D(\Acal,m,n)$, and in the next section we prove an 
elliptic analogue of Theorem~\ref{thm:hageCDAmnlogm}.
We begin by recalling some basic properties of canonical height
functions on elliptic curves.

\begin{definition}
Let~$E/K$ be an elliptic curve defined over a number field. We
write
\[
  \hhat : E(\Kbar)\longrightarrow \RR
\]
for the absolute logarithmic canonical height~\cite[VIII~\S9]{MR2514094},
and for each~$v\in M_\Kbar$ we let
\[
  \lhat_v : E(\Kbar_v)\setminus\{O\}\longrightarrow\RR
\]
be a local canonical height, normalized as described
in~\cite[Chapter~VI]{ATAEC}.
\end{definition}

\begin{proposition}
\label{prop:localhtproperties}
The local canonical height satisfies the following\textup:
\begin{parts}
\Part{(a)}
For all $v\in M_K$ there is a constant $c(v)$ such that
\[
  \l_v(P)\ge -c(v)\quad\text{for all $P\in E(\Kbar_v)$.}
\]
Further, if~$v\in M_K^0$ and~$E$ has good reduction at~$v$, then
we can take~$c(v)=0$.
\Part{(b)}
The global height is the sum of the local heights. Thus for any finite
extension $L/K$ and any $P\in E(L)\setminus\{O\}$ we have
\[
  \hhat(P) = \sum_{v\in M_L} \lhat_v(P).
\]
\end{parts}
\end{proposition}
\begin{proof}
The first part of~(a) follows from~\cite[Theorem~VI.1.1(a)]{ATAEC}, 
which says in particular that~$\lhat_v$ has a logarithmic pole
as $P\to O$ in the $v$-adic topology and that~$\l_v$ is bounded
on the complement of any $v$-adic neighborhood of~$O$. The second
part of~(a) follow from~\cite[Theorem~VI.4.1]{ATAEC}, which says that
if~$P$ reduces to a non-singular point modulo~$v$, then
\[
  \lhat_v(P) = \frac12\max\bigl\{-v\bigl(x(P)\bigr),0\bigr\}
     +\frac1{12}v(\gD_{E/K}).
\]
This quantity is clearly non-negative.  Finally,
\cite[Theorem~VI.2.1]{ATAEC} gives a proof of~(b).
\end{proof}

\begin{definition}
Given
\begin{itemize}
\item[$K/\QQ$]
a number field,
\item[$\gm$]
an integral idea of $K$ with norm $m=\Norm_{K/\QQ}\gm\ge2$,
\item[$E/K$]
an elliptic curve, and
\item[$\Pcal$]
a finite set of nontorsion points in $E(\Kbar)$,
\end{itemize}
we define
\[
  \Delta(\Pcal,\gm,n) =
  \sum_{P\in\Pcal} \frac{1}{\log m}\sum_{v\mid\gm} \frac{1}{n^2} \lhat_v(nP).
\]
This quantity is the elliptic analogue of the quantity~$\D(\Acal,m,n)$
that we defined in Section~\ref{section:prop1}.
\end{definition}

The following estimate will be used later when we do an averaging
argument. It is the analogue of Proposition~\ref{prop:DAjmngeDAmn}.

\begin{lemma}
\label{lemma:DjPgeDP}
With notation as above, assume that~$E$ has potential good reduction
at every prime dividing~$\gm$.  Let $j\ge1$ be an integer, and let
$j\Pcal=\{jP:P\in\Pcal\}$. Then
\[
  \Delta(j\Pcal,\gm,n)\ge \Delta(\Pcal,\gm, n).
\]
\end{lemma}
\begin{proof}
Replacing~$K$ by a finite extension, we may assume that~$E$ has good
reduction at all primes dividing~$\gm$.  Let~$v\in M_K^0$ be any place
at which~$E$ has good reduction, and let~$\pi_v\in K$ be a uniformizer
at~$v$. Further, let
\[
  E_0(\Kbar_v) \subset E_1(\Kbar_v) \subset E_2(\Kbar_v) \subset \cdots
\]
be the formal group filtration of~$E(\Kbar_v)$;
see~\cite[Chapters~IV,~VII]{MR2514094}. Here $E_0=E$, since we have assumed
good reduction, and~$E_1$ is the formal group.  The explicit formula
for~$\lhat_v$~\cite[Theorem~VI.4.1]{ATAEC} then has the form
\[
  \lhat_v(P) = \max\bigl\{r\ge0 : P \in E_r(\Kbar_v)\bigr\}v(\pi_v).
\]
Since the filtration consists of subgroups, it is immediate that
\[
  \lhat_v(jP) \ge \lhat_v(P)
  \qquad\text{for all $j\ge1$.}
\]
Summing over $P\in\Pcal$ and $v\mid\gm$, and dividing by $n^2\log m$,
the desired result is immediate from the definition of~$\D$.  
\end{proof}

\begin{remark}
If~$E$ has potential multiplicative reduction at~$v$, then it is
possible to have $\lhat_v(jP)<\lhat_v(P)$, so~$\D(j\Pcal,\gm,n)$ may
be strictly smaller than~$\D(\Pcal,\gm,n)$.
\end{remark}

\section{A Height Lower Bound for Points on Elliptic Curves}
\label{section:mainresult}

In this section we prove our second main result, which is an elliptic
analogue of the height lower bound given in
Theorem~\ref{thm:hageCDAmnlogm}.  We do not explicitly keep track of
the dependence on the field~$K$ or the curve~$E$, although it would be
possible to do so.  We start with a Fourier averaging estimate that is
analogous to Proposition~\ref{prop:fejarestimate} and that has been
applied in the past to the elliptic Lehmer conjecture~\cite{HiSi1}, to
Lang's height lower bound conjecture~\cite{HiSi2}, and to Arakelov
theory~\cite{LangAT}.  In order to state the result, we use the
following useful definition from~\cite{LangDG}.

\begin{definition}
Let $K/\QQ$ be a number field. An \emph{$M_K$-constant} is a map
\[
  c : M_K\longrightarrow [0,\infty)
\]
with the property that $\{v\in M_K:c(v)\ne0\}$ is a finite set.  (For
convenience, we consider only non-negative $M_K$-constants.)  A
\emph{normalized} \emph{$M_\Qbar$-constant} is a collection of
$M_K$-constants
\[
  c_K : M_K\longrightarrow\RR,
\]
one for each number field~$K/\QQ$, satisfying the compatibility
condition that for all number fields $L/K$ and all~$v\in M_K$,
\[
  \sum_{w\in M_L,\, w|v} \frac{[L_w:K_v]}{[L:K]}c_L(w) = c_K(v).
\]
\end{definition}

\begin{proposition}
\label{prop:fourieravg}
Let $E/\Qbar$ be an elliptic curve. There are normalized
$M_{\Qbar}$-constants $c_1$ and~$c_2$, depending only on~$E$, such
that for all integers~$J\ge2$, all nontorsion points~$P\in E(\Qbar)$,
and all absolute values~$v\in M_\Qbar$ we have
\begin{equation}
  \label{eqn:fejarjPgec1v}
  \sum_{j=1}^J \left(1-\frac{j}{J+1}\right)\lhat_v(jP)
  \ge -c_1(v)\log(J)-c_2(v).
\end{equation}
\textup(We may, in fact, take $c_1(v)=0$  for all nonarchimedean~$v$.\textup)
\end{proposition}
\begin{proof}
If~$v$ is nonarchimedean and~$E$ has good reduction at~$v$, then the
local height~$\lhat_v$ is non-negative, so we can take
$c_1(v)=c_2(v)=0$.  For nonarchimedean~$v$ of bad reduction, the
inequality~\eqref{eqn:fejarjPgec1v} with $c_1(v)=0$ is proven
in~\cite{HiSi2}.  Finally, for archimedean~$v$, the local height
functions are Green's functions and the desired result follows from a
general theorem of Elkies~\cite[Theorem~5.1]{LangAT} that is valid on
curves of positive genus. More precisely, Elkies' theorem says that
there is a constant $c=c(E,v)$ such that for any distinct
points~$P_0,\ldots,P_J\in E(K_v)$ we have
\begin{equation}
  \label{eqn1}
  \sum_{0\le i<j\le J}
  \lhat_v(P_j-P_i) \ge -\frac{1}{2\pi}(J+1)\log J - cJ.
\end{equation}
(We are using the fact that~$\lhat_v$ is an even function.)
Taking $P_j=jP$ for $0\le j\le J$, we find that
\begin{equation}
  \label{eqn2}
  \sum_{0\le i<j\le J} \lhat_v(P_j-P_i)
  = \sum_{0\le i<j\le J} \lhat_v\bigl((j-i)P\bigr)
  = \sum_{j=1}^J (J+1-j)\lhat_v(jP).
\end{equation}
Combining~\eqref{eqn1} and~\eqref{eqn2} and dividing by~$J+1$
gives~\eqref{eqn:fejarjPgec1v}.
\end{proof}

We have now assembled all of the tools required to prove our main 
result on elliptic curves.

\begin{theorem}
\label{thm:sumhhatPgeD}
Suppose that we are given the following quantities\textup:
\begin{notation}
\item[$K/\QQ$]
a number field.
\item[$E/K$]
an elliptic curve.
\item[$n$]
a positive integer.
\item[$\gm$]
an integral ideal of~$K$ with norm $m=\Norm_{K/\QQ}\gm\ge2$.
\item[$\Pcal$]
a finite set of nontorsion points in $E(\Kbar)$.
\end{notation}
Suppose further that~$E$ has potential good reduction at every prime
dividing~$\gm$. Then there is a constant $C_E$, depending only on~$E$,
such that for all integers~$J\ge1$ we have
\[
  \sum_{P\in\Pcal} \hhat(P) \ge
  \frac{6}{(J+1)(J+2)}\left(\D(\Pcal,\gm,n)\log m -C_E
  \frac{|\Pcal|}{n^2}\cdot\frac{\log(J+1)}{J}\right).
\]
\end{theorem}
\begin{proof}
To ease notation, we let
\[
  f_j = 1-\frac{j}{J+1}
  \qquad\text{and}\qquad
  F_k = \sum_{j=1}^J j^kf_j.
\]
Earlier we used the values of~$F_0$ and~$F_1$. In this section we will
use the values
\begin{equation}
  \label{eqn:F0F1F2}
  F_0 = \frac{J}{2}
  \qquad\text{and}\qquad
  F_2 = \frac{J(J+1)(J+2)}{12}.
\end{equation}
\par
Replacing~$K$ by a finite extension, we may assume that~$\Pcal\subset
E(K)$.  We let
\[
  M_K^\bad = M_K^\infty \cup \{v\in M_K^0 : 
   \text{$E$ has bad reduction at $v$}\} .
\]
Then
\[
  v\in M_K\setminus M_K^\bad
  \quad\Longrightarrow\quad
  \l_v(Q)\ge0\quad\text{for all $Q\in E(\Kbar_v)$.}
\]
\par
We compute
\begin{align*}
  n^2\sum_{P\in\Pcal} \hhat(P)
  & = \sum_{P\in\Pcal} \hhat(nP) \\
  & = \sum_{P\in\Pcal} \sum_{v\in M_K} \lhat_v(nP) \\
  &\ge \sum_{P\in\Pcal} \sum_{v\mid\gm} \lhat_v(nP) 
       + \sum_{P\in\Pcal} \sum_{v\in M_K^\bad} \lhat_v(nP) \\
  &= (n^2\log m)\D(\Pcal,\gm,n)
       + \sum_{P\in\Pcal} \sum_{v\in M_K^\bad} \lhat_v(nP).
\end{align*}
Replacing~$\Pcal$ with $j\Pcal=\{jP:P\in\Pcal\}$ and using
Lemma~\ref{lemma:DjPgeDP}, which says that
$\D(j\Pcal,\gm,n)\ge\D(\Pcal,\gm,n)$ (this is where we use the
assumption that~$E$ has potential good reduction at the primes
dividing~$\gm$), we find that
\[
  n^2j^2\sum_{P\in\Pcal} \hhat(P)
  \ge (n^2\log m)\D(\Pcal,\gm,n)
       + \sum_{P\in\Pcal} \sum_{v\in M_K^\bad} \lhat_v(njP).
\]
Multiplying both sides by~$f_j$ and summing $j=1$ to~$J$ gives
\begin{align*}
  F_2n^2\sum_{P\in\Pcal}& \hhat(P)\\
  &\ge F_0(n^2\log m)\D(\Pcal,\gm,n)
       + \sum_{j=1}^J \sum_{P\in\Pcal} \sum_{v\in M_K^\bad} f_j\lhat_v(njP) \\
  &\ge F_0(n^2\log m)\D(\Pcal,\gm,n)
       + |\Pcal|\sum_{v\in M_K^\bad} \inf_{Q\in E(K)} \sum_{j=1}^J f_j\lhat_v(jQ).
\end{align*}
Proposition~\ref{prop:fourieravg} says that there
are~$M_K$-constants~$c_1$ and~$c_2$, depending only on~$E$, such that
\[
  \inf_{Q\in E(K)} \sum_{j=1}^J f_j\lhat_v(jQ) \ge -c_1(v)\log(J)-c_2(v).
\]
Summing over $v\in M_K^\bad$ gives constants that depend only on~$E$,
so adjusting the constants and using the assumption that $J\ge1$, we
find that there is a constant~$C_E$, depending only on~$E$, such that
\[
  F_2n^2\sum_{P\in\Pcal} \hhat(P)
  \ge F_0(n^2\log m)\D(\Pcal,\gm,n)
       -C_E |\Pcal|\log(J+1).
\]
Using the formulas~\eqref{eqn:F0F1F2} for~$F_0$ and~$F_2$, dividing
by~$F_2n^2$, and adjusting the constant gives the desired result.
\end{proof}

Corollary~\ref{cor:htgeelogm} says roughly that the classical Lehmer's
conjecture is true for polynomials~$f(X)$ such that
\begin{equation}
  \label{eqn:Fn1XdivfinZmZX}
  \text{$\F_{n-1}(X)$ divides $f(X)$ in $(\ZZ/m\ZZ)[X]$ for some
    $n\ge\e\deg(f)$.}
\end{equation}
As noted in Remark~\ref{remark:DAfmD1geDD1}, the divisibility
condition in~\eqref{eqn:Fn1XdivfinZmZX} is stronger than the assertion
that $\D(\Acal_f,m,n)\ge(n-1)/n$, where~$\Acal_f$ denotes the set of
roots of~$f$. Since we assume that~$n\ge2$, this implies in particular
that $\D(\Acal_f,m,n)$ is uniformly bounded away from~$0$.  Thus the
following result is an elliptic version of a strengthening of
Corollary~\ref{cor:htgeelogm}.

\begin{definition}
Let $E/K$ be an elliptic curve and let $Q\in E(\Kbar)$. We let
\[
  \Pcal_Q = \{\s(Q) : \s\in\Gal(\Kbar/K)\}
  \quad\text{and}\quad
  D_Q = [K(Q):K] = |\Pcal_Q|.
\]
We remark that all of the points in~$\Pcal_Q$ have the same canonical
height; cf.~\cite[Theorem~VIII.5.10]{MR2514094}.
\end{definition}

\begin{corollary}
\label{cor:ellipticlehmer}
Let~$E/K$ be an elliptic curve defined over a number field.  Fix
constants $\d,\e>0$. Then the elliptic Lehmer
conjecture~\eqref{eqn:ellipticlehmer} is true for the
following set of points\textup:
\begin{equation}
  \label{eqn:QEKbarDd}
  \left\{ Q \in E(\Kbar) : 
   \parbox{.65\hsize}{\noindent
    there exists an integral ideal $\gm$ of $K$ 
    with $\Norm_{K/\QQ}\gm\ge2$ such that $E$ has
    good reduction at all primes dividing $\gm$ and
    an integer $n\ge\max\{\sqrt{\e D_Q},2\}$ 
    with $\D(\Pcal_Q,\gm,n)\ge\d$
  }
  \right\}.
\end{equation}
For points in the set~\eqref{eqn:QEKbarDd}, the constant
in~\eqref{eqn:ellipticlehmer} will have the form
$C_{E/K,\d,\e}\log{m}$, where $C_{E/K,\d,\e}$ is positive and depends
only on the indicated quantities.
\end{corollary}
\begin{proof}
We are given that $n\ge\sqrt{\e D_Q}$ and $\D(\Pcal_Q,\gm,n)\ge\d$.
Using these values in Theorem~\ref{thm:sumhhatPgeD} together with
some trivial estimates yields
\[
  \sum_{P\in\Pcal_Q}\hhat(P)
  \ge \frac{1}{J^2}\left(\d\log m
           -\frac{C_E}{\e}\cdot\frac{\log(J+1)}{J}\right).
\]
We now choose~$J$ to be the smallest integer satisfying
\[
  \frac{\log(J+1)}{J} \le \min\left\{ \frac{\e\d}{2C_E}\log m, \frac12\right\}.
\]
This yields  an estimate of the desired form
\[
  D_Q\hhat(Q) \ge C_{E/K,\d,\e}\log m,
\]
where we are using the fact, noted earlier, that every point in~$\Pcal_Q$
has canonical height equal to~$\hhat(Q)$.
\end{proof}

\section{Other Congruence Conditions on the Coefficients}
\label{section:coefcong}

Cyclotomic polynomials play a key role in Lehmer's conjecture, so the
congruence condition~\eqref{eqn:feqaop} and the more general
congruence divisibility condition~\eqref{eqn:fXdivXn1X1ZmZ} are
natural ones to consider. However, there is no reason not to consider
similar congruences in which the cyclotomic polynomial is replaced by
some other polynomial.  This was done in considerable generality by
Samuels~\cite{MR2313990}. To illustrate, we reprove a special
case of one of Samuels' result and use it to make some remarks.

\begin{definition}
The \emph{length} of a polynomial $g(X)=\sum a_iX^i\in\ZZ[X]$ is the
quantity
\[
  L(g) = \sum |a_i|.
\]
\end{definition}

\begin{theorem}
\label{thm:othercongruences}
\textup{(Special Case of \cite[Corollary~5.3]{MR2313990})}
Let $m\ge2$ and let~$f(X)\in\ZZ[X]$ be a monic polynomial of
degree~$D$ satisfying $f(1)\ne0$. Further let $u(X)\in\ZZ[X]$
be a polynomial of degree at most~$D-1$, and suppose that
\[
  f(X)\equiv \F_D(X) + u(X) \pmod{m},
\]
but that~$f(X)$ has no roots in common with $\F_D(X)+u(X)$. Then
\begin{align*}
  \sum_{f(\a)=0} h(\a) 
  &\ge \frac{D}{D+1}
            \log\left(\frac{m}{L\bigl(X^{D+1}-1+(X-1)u(X)\bigr)}\right).
\end{align*}
\end{theorem}

\begin{proof}
We are given that
\[
  f(X) = \F_D(X) + u(X) + mr(X)
  \qquad\text{for some $r(X)\in\ZZ[X]$.}
\]
Then
\begin{align*}
  \Resultant\bigl(f(X)&,X^{D+1}-1+(X-1)u(X)\bigr)\\*
  &= \Resultant\bigl(f(X),\F_D(X)+u(X)\bigr)
     \cdot\Resultant\bigl(f(X),X-1\bigr) \\
  &= \Resultant\bigl(mr(X),\F_D(X)+u(X)\bigr)f(1) \\
  &= m^D\Resultant\bigl(r(X),\F_D(X)+u(X)\bigr)f(1).
\end{align*}
By assumption, the resultants are nonzero integers, so we find that
\begin{align}
  \label{eqn:DlogmaD1DLX}
  D\log m
  &\le \log\left|\Resultant\bigl(f(X),X^{D+1}-1+(X-1)u(X)\bigr)\right| 
    \notag\\
  &= \sum_{f(\a)=0} \sum_{v\in M_K^\infty} \log\|\a^{D+1}-1+(\a-1)u(\a)\|_v
    \notag\\
  &= \sum_{v\in M_K^\infty}\sum_{\substack{f(\a)=0\\\|\a\|_v>1\\}} \log\|\a^{D+1}\|_v \notag\\
  &\qquad{}  + \sum_{v\in M_K^\infty}\sum_{\substack{f(\a)=0\\\|\a\|_v>1\\}} 
            \log\left\|\frac{\a^{D+1}-1+(\a-1)u(\a)}{\a^{D+1}}\right\|_v \notag\\
  &\qquad{}  + \sum_{v\in M_K^\infty}\sum_{\substack{f(\a)=0\\\|\a\|_v\le1\\}} 
                  \log\|\a^{D+1}-1+(\a-1)u(\a)\|_v  \notag\\
  &\le h(\a^{D+1}) + D\sup_{|z|=1}\log\bigl|z^{D+1}-1+(z-1)u(z)\bigr| \\
  &\le (D+1)h(\a) + D\log L\bigl(X^{D+1}-1+(X-1)u(X)\bigr). \notag
\end{align}
(We note that in \eqref{eqn:DlogmaD1DLX}, it suffices to take the
supremum over~$|z|=1$, since $\log|w|$ is harmonic inside the unit
disk.)
\end{proof}

\begin{remark}
The upshot of Theorem~\ref{thm:othercongruences} is that if~$m$ is
sufficiently large, then we obtain a Lehmer-type lower bound. However,
in the cyclotomic case, i.e., $u=0$, a Fourier averaging argument allowed us to
prove nontrivial estimates for all values of~$m\ge2$.  We do not know
how to perform such an averaging argument in the general case. It
would also be interesting to prove a version of
Theorem~\ref{thm:othercongruences} under the weaker assumption
that~$f(X)$ is divisible modulo~$m$ by~$\F_{n-1}(X)$ for, say, $n\ge\e
D$. Again we do not have the requisite averaging lemma.
\end{remark}

\begin{remark}
If we take $u=0$ in  Theorem~\ref{thm:othercongruences}, we obtain
the estimate
\[
  \sum_{f(\a)=0} h(\a) \ge \frac{D}{D+1}\log\frac{m}{2}.
\]
This has the same form as Theorem~\ref{thm:lehmodm}, although the
constant in Theorem~\ref{thm:lehmodm} is a little bit better than ours
(and our estimate is trivial for $m=2$).  On the other hand, it is
interesting that such an elementary argument produces a lower bound
that agrees with the best known lower
bounds~\cite{BoDoMo,BoHaMo,DuMo,GaIsMoPiWi} up to an
additional~$O(1/m^2)$.
\end{remark}

\begin{remark}
The estimate proven in Theorem~\ref{thm:othercongruences} is
nontrivial if and only if $m>L\bigl(X^{D+1}-1+(X-1)u(X)\bigr)$.  As
Samuels does in~\cite{MR2313990}, it is sometimes possible to improve
the estimate a little bit. We illustrate with the case $u(X)=-2$, so
\[
  f(X)\equiv X^D+X^{D-1}+\cdots+X^2+X-1\pmod{m}
\]
and
\[
  L\bigl(X^{D+1}-1-(X-1)u(X)\bigr)  =  L(X^{D+1}-2X+1)=4.
\]
This gives a nontrivial height bound for $m\ge5$.  If~$D$ is odd,
then the supremum in~\eqref{eqn:DlogmaD1DLX} occurs at~$z=-1$ and is
equal to~$\log4$, but if~$D$ is even, then the supremum is strictly
smaller than~$\log4$ and we can obtain a small improvement in the
theorem.  However, for large (even) values of~$D$ we have
\[
   \sup_{|z|=1}  |z^{D+1}-2z+1| = 4+O(D^{-2})
  \quad \text{at $z\approx -e^{\pi i/(D+1)}$,}
\]
so for $m=4$ we only obtain $\sum h(\a) \gg D^{-2}$, which is weaker
than Lehmer's conjecture.
\end{remark}

\section{Proof of Proposition~\ref{prop:fejarestimate}}
\label{section:fejarproof}

In this section we give the proof of
Proposition~\ref{prop:fejarestimate}, for which we need the
following standard lemma.

\begin{lemma}
\label{lemma:Gttheta}
For all $\theta\in\RR\setminus2\pi i\ZZ$ and all $t\ge0$
we have
\[
  \log|1-e^{i\theta}| \le \log|1-e^{-t}e^{i\theta}| + \frac12t.
\]
\end{lemma}
\begin{proof}
For notational convenience we let
\[
  F_t(\theta) =
  \log|1-e^{-t}e^{i\theta}| + \frac12t - \log|1-e^{i\theta}|,
\]
so we need to prove that $F_t(\theta)\ge0$.  We have
\[
  F_0(\theta)=0
  \qquad\text{for all $\theta\in\RR\setminus2\pi i\ZZ$.}
\]
For $t>0$ we observe that
\[
  \log|1-e^{-t}e^{i\theta}|
   = \Re\biggl(\sum_{n=1}^\infty -\frac{e^{-nt}e^{in\theta}}{n}\biggr)
\]
is given by an absolutely convergent series, so we may differentiate it
term-by-term. Hence
\begin{align*}
  \frac{\partial F_t}{\partial t}(\theta)
  &= \Re\biggl(\sum_{n=1}^\infty (e^{-t+i\theta})^n\biggr) + \frac12 
  = \Re\left( \frac{e^{-t+i\theta}}{1-e^{-t+i\theta}}\right) + \frac12 \\
  &= \frac{e^{2t}-1}{(e^t-\cos\theta)^2+\sin^2\theta} 
   > 0 \qquad\text{for all $t>0$.}
\end{align*}
Thus for any fixed $\theta\in\RR\setminus2\pi i\ZZ$, the
function~$F_t(\theta)$ as a function of~$t\ge0$
satisfies~$F_0(\theta)=0$ and~$(\partial F_t/\partial t)(\theta)\ge
0$. Hence~$F_t(\theta)\ge0$ for all $t\ge0$. 
\end{proof}

{\tiny
Here is the elementary algebraic verification of the trigonometric
identity that was used in the above calculation.
\begin{align*}
  \Re\left( \frac{e^{-t+i\theta}}{1-e^{-t+i\theta}}\right) + \frac12 
  &= \Re\left(\frac{e^{i\theta}}{e^t-e^{i\theta}}\right) + \frac12  
  = \Re\left(\frac{e^{i\theta}(e^t-e^{-i\theta})}{e^{2t}-2(\cos\theta)e^t+1}\right)
          + \frac12 \\
  &= \Re\left(\frac{e^{i\theta}e^t-1}{e^{2t}-2(\cos\theta)e^t+1}\right)
          + \frac12 
  = \frac{e^t\cos\theta-1}{e^{2t}-2(\cos\theta)e^t+1}+\frac12 \\
  &= \frac{e^{2t}-1}{e^{2t}-2(\cos\theta)e^t+1} 
  = \frac{e^{2t}-1}{(e^t-\cos\theta)^2+\sin^2\theta}.
\end{align*}
}

\begin{proof}[Proof of Proposition~$\ref{prop:fejarestimate}$]
The functions $\log|1-z^j|$ are harmonic on the open unit disk $|z|<1$,
so the maximum occurs on the boundary. For $z=e^{i\theta}$ on the unit
circle we estimate
\begin{align*}
  \sum_{j=1}^J&\left(1-\frac{j}{J+1}\right)\log|1-z^j|\\
  &\le \sum_{j=1}^J \left(1-\frac{j}{J+1}\right) 
                    \left(\log|1-e^{-t}e^{ij\theta}| + \frac12t\right)
   &&\text{from Lemma \ref{lemma:Gttheta},} \\
  &= \Re\biggl(
      \sum_{j=1}^J \left(1-\frac{j}{J+1}\right) 
                       \sum_{k=1}^\infty -\frac{e^{-kt}e^{ijk\theta}}{k} \biggr)
     + \frac{Jt}{4} \\
  &= \Re\biggl(
      \sum_{k=1}^\infty \frac{e^{-kt}}{k}\sum_{j=1}^J 
          -\left(1-\frac{j}{J+1}\right)e^{ijk\theta} \biggr)+ \frac{Jt}{4} \\
  &=       \sum_{k=1}^\infty \frac{e^{-kt}}{k}
           \biggl(\frac12
           -\frac{1}{2J+2}\biggl|\sum_{j=0}^Je^{ijk\theta}\biggr|^2\biggr)
     + \frac{Jt}{4} \hidewidth\\
  &\ge \sum_{k=1}^\infty \frac{e^{-kt}}{2k} + \frac{Jt}{4} \\
  &= -\frac12\log(1-e^{-t}) + \frac{Jt}{4}.
\end{align*}
This estimate holds for all $t>0$, so in particular we can
set $t=\log(1+2J^{-1})$, which (after some algebra) gives the estimate
\[
  \sum_{j=1}^J f_j\log|1-z^j|
  \le \frac12\log\left(\frac{J}{2}+1\right)
      +\frac{J}{4}\log\left(1+\frac{2}{J}\right).
\]
Finally we observe that $x\log(1+x^{-1})\le 1$ for all $x\ge0$, which
gives the desired result.
\end{proof}




\begin{thebibliography}{10}


\bibitem{AmDv}
F.~Amoroso and R.~Dvornicich.
\newblock A lower bound for the height in abelian extensions.
\newblock {\em J. Number Theory}, 80(2):260--272, 2000.

\bibitem{BlMo}
P.~E. Blanksby and H.~L. Montgomery.
\newblock Algebraic integers near the unit circle.
\newblock {\em Acta Arith.}, 18:355--369, 1971.

\bibitem{BoDoMo}
P.~Borwein, E.~Dobrowolski, and M.~J. Mossinghoff.
\newblock Lehmer's problem for polynomials with odd coefficients.
\newblock {\em Ann. of Math. (2)}, 166(2):347--366, 2007.

\bibitem{BoHaMo}
P.~Borwein, K.~G. Hare, and M.~J. Mossinghoff.
\newblock The {M}ahler measure of polynomials with odd coefficients.
\newblock {\em Bull. London Math. Soc.}, 36(3):332--338, 2004.

\bibitem{Dob}
E.~Dobrowolski.
\newblock On a question of {L}ehmer and the number of irreducible factors of a
  polynomial.
\newblock {\em Acta Arith.}, 34:391--401, 1979.

\bibitem{DuMo}
A.~Dubickas and M.~J. Mossinghoff.
\newblock Auxiliary polynomials for some problems regarding {M}ahler's measure.
\newblock {\em Acta Arith.}, 119(1):65--79, 2005.

\bibitem{HiSi1}
M.~Hindry and J.~H. Silverman.
\newblock The canonical height and integral points on elliptic curves.
\newblock {\em Invent. Math.}, 93(2):419--450, 1988.

\bibitem{HiSi2}
M.~Hindry and J.~H. Silverman.
\newblock On {L}ehmer's conjecture for elliptic curves.
\newblock In {\em S\'eminaire de Th\'eorie des Nombres, Paris 1988--1989},
  volume~91 of {\em Progr. Math.}, pages 103--116. Birkh\"auser Boston, Boston,
  MA, 1990.

\bibitem{HiSi}
M.~Hindry and J.~H. Silverman.
\newblock {\em Diophantine {G}eometry: An {I}ntroduction}, volume 201 of {\em
  Graduate Texts in Mathematics}.
\newblock Springer-Verlag, New York, 2000.

\bibitem{GaIsMoPiWi}
M.~I.~M. Ishak, M.~J. Mossinghoff, C.~Pinner, and B.~Wiles.
\newblock Lower bounds for heights in cyclotomic extensions.
\newblock {\em J. Number Theory}, 130(6):1408--1424, 2010.

\bibitem{LangDG}
S.~Lang.
\newblock {\em Fundamentals of {D}iophantine {G}eometry}.
\newblock Springer-Verlag, New York, 1983.

\bibitem{LangAT}
S.~Lang.
\newblock {\em Introduction to {A}rakelov {T}heory}.
\newblock Springer-Verlag, New York, 1988.

\bibitem{lang:algebra}
S.~Lang.
\newblock {\em Algebra}, volume 211 of {\em Graduate Texts in Mathematics}.
\newblock Springer-Verlag, New York, third edition, 2002.

\bibitem{Laur}
M.~Laurent.
\newblock Minoration de la hauteur de {N}\'eron-{T}ate.
\newblock In {\em S\'eminaire de Th\'eorie des Nombres}, Progress in
  Mathematics, pages 137--151. Birkh\"auser, 1983.
\newblock Paris 1981--1982.

\bibitem{Leh}
D.~H. Lehmer.
\newblock Factorization of certain cyclotomic functions.
\newblock {\em Ann. of Math. (2)}, 34(3):461--479, 1933.

\bibitem{Mass}
D.~W. Masser.
\newblock Counting points of small height on elliptic curves.
\newblock {\em Bull. Soc. Math. France}, 117(2):247--265, 1989.

\bibitem{MR2313990}
C.~L. Samuels.
\newblock The {W}eil height in terms of an auxiliary polynomial.
\newblock {\em Acta Arith.}, 128(3):209--221, 2007.

\bibitem{MR2496463}
C.~L. Samuels.
\newblock Estimating heights using auxiliary functions.
\newblock {\em Acta Arith.}, 137(3):241--251, 2009.

\bibitem{ATAEC}
J.~H. Silverman.
\newblock {\em Advanced {T}opics in the {A}rithmetic of {E}lliptic {C}urves},
  volume 151 of {\em Graduate Texts in Mathematics}.
\newblock Springer-Verlag, New York, 1994.

\bibitem{MR2514094}
J.~H. Silverman.
\newblock {\em The {A}rithmetic of {E}lliptic {C}urves}, volume 106 of {\em
  Graduate Texts in Mathematics}.
\newblock Springer, Dordrecht, second edition, 2009.

\bibitem{Smy2}
C.~Smyth.
\newblock The {M}ahler measure of algebraic numbers: a survey.
\newblock In {\em Number theory and polynomials}, volume 352 of {\em London
  Math. Soc. Lecture Note Ser.}, pages 322--349. Cambridge Univ. Press,
  Cambridge, 2008.

\bibitem{Smy1}
C.~J. Smyth.
\newblock On the product of the conjugates outside the unit circle of an
  algebraic integer.
\newblock {\em Bull. London Math. Soc.}, 3:169--175, 1971.

\bibitem{Stew}
C.~L. Stewart.
\newblock Algebraic integers whose conjugates lie near the unit circle.
\newblock {\em Bull. Soc. Math. France}, 106(2):169--176, 1978.

\end{thebibliography}

\end{document}